\documentclass[11pt]{amsart}

\usepackage{amssymb, amsthm, amsmath}
\newtheorem{theorem}{Theorem}[section]

\newtheorem{lemma}[theorem]{Lemma}
\newtheorem{corollary}[theorem]{Corollary}
\newtheorem{prop}[theorem]{Proposition}

\newtheorem*{theorem*}{Theorem}{\bf}{\it}
\newtheorem*{proposition*}{Proposition}{\bf}{\it}
\newtheorem*{observation*}{Observation}{\bf}{\it}
\newtheorem*{lemma*}{Lemma}{\bf}{\it}

\theoremstyle{definition}

\theoremstyle{remark}
\newtheorem{remark}[theorem]{Remark}
\newcommand{\closure}[2][3]{%
  {}\mkern#1mu\overline{\mkern-#1mu#2}}

\newcommand{\Z}{\mathbb Z}
\newcommand{\R}{\mathbb R}
\newcommand{\D}{\mathbb D}

\newcommand{\dist}{{\rm{dist}}}
\newcommand{\tu}{\tilde{u}}
\newcommand{\tv}{\tilde{v}}
\newcommand{\tw}{\tilde{w}}
\newcommand{\tf}{\tilde{f}}

\newcommand{\tV}{\tilde{V}}

\def\Xint#1{\mathchoice
{\XXint\displaystyle\textstyle{#1}}%
{\XXint\textstyle\scriptstyle{#1}}%
{\XXint\scriptstyle\scriptscriptstyle{#1}}%
{\XXint\scriptscriptstyle\scriptscriptstyle{#1}}%
\!\int}
\def\XXint#1#2#3{{\setbox0=\hbox{$#1{#2#3}{\int}$ }
\vcenter{\hbox{$#2#3$ }}\kern-.6\wd0}}

\def\dashint{\Xint-}
\begin{document}
\title[Ratios of harmonic functions]{Ratios of harmonic functions with the same zero set}
\thanks{The main result, Theorem \ref{H}, and Sections \ref{se:4}, \ref{se:5}, \ref{se:6} were supported by the Russian Science Foundation grant 14-21-00035.
Theorem \ref{main} and Sections \ref{se:2}, \ref{se:3} were supported by Project 213638 of the Research Council of Norway.}
\author{Alexander Logunov}
\address{Chebyshev Laboratory, St.Petersburg State University, 14th Line 29B, Vasilyevsky Island, St.Petersburg 199178, Russia}
\address{School of Mathematical Sciences, Tel Aviv University Tel Aviv 69978, Israel}
\email{log239@yandex.ru}
\author{Eugenia Malinnikova}
\address{Department of Mathematical Sciences,
Norwegian University of Science and Technology
7491, Trondheim, Norway}
\email{eugenia@math.ntnu.no}

\begin{abstract}
  We study the ratio of harmonic functions $u,v$  which have the same zero set $Z$ in the unit ball $B\subset \mathbb{R}^n$. The ratio $f=u/v$ can be extended to a real analytic nowhere vanishing function in $B$. We prove the Harnack inequality and the gradient estimate for such ratios in any dimension: for a given compact set $K\subset B$ we show that $\sup_K|f|\le C_1\inf_K|f|$ and $\sup_K\left|\nabla f\right|\le C_2 \inf_K|f|$, where $C_1$ and $C_2$ depend on $K$ and $Z$ only. In dimension two we specify the dependence of the constants on $Z$ in these inequalities by showing that only the number of nodal domains of $u$, i.e. the number of connected components of $B\setminus Z$, plays a role. 
\end{abstract}
\maketitle


\section{Introduction}
\subsection{Ratios of harmonic functions and Harnack's inequalities}
 Let $u$ and $v$ be real-valued harmonic functions in a domain $\Omega\subset\R^n$. Suppose that the zero sets of $u$ and $v$ coincide: $Z(u)=Z(v)=Z$. Then one may consider the ratio $f=u/v$. It was conjectured by Dan Mangoubi \cite{M} that such ratios and their gradients satisfy the following  Harnack inequalities,
\begin{equation}\label{eq:H}
\sup_K|f|\le C_1\inf_K|f|,
\end{equation}
 and 
\begin{equation}\label{eq:GE}
\sup_K\left|\nabla f\right|\le C_2 \inf_K|f|
\end{equation}
where $K$ is a compact subset of $\Omega$ and the constants $C_1, C_2$ depend  on $K$ and the nodal set $Z$ only. 
The inequalities (\ref{eq:H}) and (\ref{eq:GE}) follow from the classical Harnack principle when $Z=\emptyset$. They were proved by Mangoubi in dimension two \cite{M} and then by the authors in dimension three \cite{LM}. In the present work we generalize the result to higher dimensions and refine the information of the constants in the above inequalities in dimension two. Connections of these inequalities to the boundary Harnack principle for harmonic functions were  discussed in \cite{M,LM}.

It was proved  in \cite{LM} that if $f$ is the ratio of two harmonic functions in $\Omega$ with the common zero set $Z$, then $f$, defined originally on $\Omega\setminus Z$, is the trace of a real analytic function in $\Omega$ that does not vanish and therefore has a constant sign in $\Omega$ (in the sequel we refer to this continuation as $f$). Furthermore, the maximum and minimum principles hold for $f$. This is not surprising, since the ratio of two harmonic functions is a solution of an elliptic equation (see \cite{M}, \cite{LM}), however, since this equation is highly degenerate, general known results are not applicable.

\subsection{Main results} The present work contains two independent results.
 First, we answer one of the questions posed in \cite{M}, by showing that in dimension two the constants in \eqref{eq:H} and \eqref{eq:GE}  depend  on the number of the  nodal domains, i.e. the number of connected components of $B\setminus Z$, only. Equivalently, we may say that the constants depend on the length of the nodal set only, see Remark \ref{r:length} below.    
\begin{theorem}\label{main}
 Let $u$ and $v$ be harmonic functions in the unit disc $\D \subset \mathbb{R}^2$ such that $Z(u)=Z(v)$ and suppose the number of nodal domains of $u$ (and $v$) is less than a fixed number $N$. 
 Let $f$ be the ratio of $u$ and $v$, then for any compact set $K \subset \D$ there exist constants $C_1=C_1(K,N)$ and $C_2=C_2(K,N)$ depending on $K$ and $N$ only  such that (\ref{eq:H}) and (\ref{eq:GE}) hold. 
\end{theorem}
The proof uses some kind of compactness principle for harmonic functions with a bounded number of nodal domains. The principle holds in dimension two only and was proved by N.~Nadirashvili \cite{N98}. However, we don't see how the estimates (\ref{eq:H}) and (\ref{eq:GE}) with uniform constants would follow immediately from this principle. We use a structure theorem for analytic functions taking real values on a fixed set, information about the critical set,  as well as estimates from the local division argument in \cite{LM}, to complete the proof.

 Our main result gives the affirmative answer to another question of Mangoubi, \cite{M}. It contains the Harnack inequality (\ref{eq:H}) and the gradient estimate (\ref{eq:GE}) for the ratios of harmonic functions in any dimension as well as estimates for all partial derivatives of the ratios. We look at families of harmonic functions with common zeros and use the following notation. Let $B$ be the unit ball in $\mathbb{R}^n$ and $Z$ be its subset, we define 
 $$H_Z:=\{u: B \to\R: \Delta u=0, Z(u)=Z\}.$$
\begin{theorem} \label{H}
There exist constants $A=A(Z)>0$ and $R=R(Z)>0$ such that  for any $u,v \in H_Z$  and any multiindex $\alpha\in\Z_+^n$ the ratio $f=u/v$ satisfies
$$ \sup\limits_{B_{1/2}}|D^\alpha f| \leq \alpha! A R^{|\alpha|}\inf\limits_{B_{1/2}}|f|.$$
\end{theorem}
  This theorem was proved for the three dimensional space in \cite{LM}. The argument therein  employed the boundary Harnack principle and the structure of the nodal sets of harmonic functions,  the latter becomes more complicated with the growth of the dimension and it is not clear if that proof can be generalized to higher dimensions. We suggest another approach here. 

 The main ingredients of the proof of Theorem \ref{H} include doubling constants for harmonic functions, the \L ojasiewicz exponents,  and some known techniques of potential theory; we refer in particular to Lemma 8.7.10 in \cite{AG}. Any mention of topology of the nodal set and  the boundary Harnack principle is avoided.

One result from \cite{LM} will be required in proofs of both theorems, we cite it here and will refer  to it as to the local division principle. First we note that if $v\in H_Z$ and $x_0\in Z$, then there exists a homogeneous harmonic polynomial $p=p(x_0,Z)$ of degree $k$ such that the Taylor expansion of  $v$ at $x_0$ is given by  
\[
v(x)=c_vp(x-x_0)+\sum_{|\alpha|>k}(\alpha!)^{-1}D^\alpha v(x_0)(x-x_0)^\alpha,\quad c_v\neq 0.\]
 The polynomial $p$ is the same for all $v \in H_Z$ (see Lemma 2.1 and Lemma 2.2 in \cite{LM}).

\begin{lemma} \label{lda}
 Let $u$ and $v$ be non-zero harmonic functions in the unit ball $B$ of $\R^n$ such that $Z(u)=Z(v)=Z$ and $0\in Z$ and let $f=u/v$ and $p=p(0,Z)$ as above. Suppose that $|c_v|>\varepsilon$, 
and $|D^\alpha  u|(0) \leq AR^{|\alpha|} \alpha!$, $ |D^\alpha  v|(0) \leq AR^{|\alpha|} \alpha! $ for any multi-index $\alpha \geq 0$.
 Then there exist $c,C>0$ and $r, \rho>0$ depending on $A$,$R$ and $\varepsilon$ only, such that
\begin{equation}\label{eq:tay}
|D^\alpha f(0)|\le Cr^\alpha \alpha!, 
\end{equation}
\begin{equation}\label{eq:lda} \sup\limits_{\rho B} |f| \leq c \inf\limits_{\rho B} |f|.\end{equation}
\end{lemma}

 See Lemma 2.4 and Lemma 2.3 in  \cite{LM}  for the proof.
\subsection*{Structure of the paper}
  We collect auxiliary information essential for the proof of Theorem \ref{main}  in Section \ref{se:2}. First,  we formulate the Nadirashvili compactness principle for harmonic functions with a bounded number of nodal domains.  Then we expose some structure results on harmonic functions sharing the same zero set. The proof of Theorem \ref{main} is given in Section \ref{se:3}. Section \ref{se:4} contains preliminary results on harmonic functions in higher 
dimensions, including classical inequalities, doubling constant 
techniques and the \L ojasiewicz exponents. These results are used in Section \ref{se:5},  where a compactness principle for harmonic functions sharing the zero 
set is established. Combining this principle with the local division 
argument,  we prove Theorem \ref{H} in Section \ref{se:6}. Some comments and open questions are given at the end of Sections 3 and 6.

\section{Toolbox for dimension two} \label{se:2}
\subsection{Compactness principle for harmonic functions with a bounded number of nodal domains}\label{ss:cp}
  The following form of compactness principle holds
 \begin{lemma}[Nadirashvili] \label{N2}
 Let $u_n$ be a sequence of harmonic functions in $\D$ and let $N \in \mathbb{N} $. Suppose that the number of nodal domains of each $u_n$ is less than $N$. Then there exist a subsequence $u_{n_k}$, a sequence $\alpha_{n_k}$  of real numbers and a non-zero function $u$ such that $\alpha_{n_k} u_{n_k}$ converge to $u$ uniformly on compact subsets of $\D$. Clearly, $u$ is harmonic in $\D$. 
 \end{lemma} 
The first step in the proof of Lemma \ref{N2} is to show that the bound on the number of nodal domains implies a bound on the number of sign changes on the boundary circle, see \cite[Section 3.4]{N98}, then one may refer to an old result of M.~S.~Robertson \cite{R39} or follow the lines of \cite{N91} and \cite{N98}. In what follows we write $f_n\rightrightarrows f$ for uniform convergence on compact subsets.
 
\begin{lemma}\label{l:nodal} 
Let $\{u_n\}$ and $\{v_n\}$ be sequences of harmonic functions in $\D$ such that $Z(u_n)=Z(v_n), u_n=f_nv_n, f_n>0$ and $u_n\rightrightarrows u$, $v_n\rightrightarrows v$ in $\D$, where $u$ and $v$ are non-zero functions. Then $Z(u)=Z(v)$.
\end{lemma}
\begin{proof}
Suppose that $u(z_0)>0$ at some point $z_0 \in \D$. Then  $u(z)> \varepsilon$ in some neighborhood of $z_0$ and $u_n(z)>\varepsilon/2$ for all sufficiently large $n$. We assumed that $f_n >0$, hence $v_n \geq 0$ in some neighborhood of $z_0$ for all sufficiently large $n$. It implies that $v \geq 0 $ in some neighborhood of $z_0$. Since $v$ is non-zero harmonic function, we conclude that $v(z_0)>0$. 
 Analogously $u(z_0)<0$ implies $v(z_0)<0$ and $v(z_0)>0$ implies $u(z_0)>0$. Thus $Z(u)=Z(v)$.
\end{proof}
\subsection{The Schwarz reflection principle and a structure result}
 Suppose that $U$ is an analytic function in $\D$ such that $\Im U (w)=0$ if and only if $\Im(w^k)=0$. Let $\epsilon$ be the $k$-th root of unity, then by the Schwarz reflection principle $U( w)= \overline{U(\bar w)}$ and $U(w\epsilon)=U(w)$. The last observation implies that coefficients $a_j$ of Taylor series of $U$ at $0$ are real and $a_j=0$ if $j$ is not divisible by $k$. Then $U=g(w^k)$, where $g$ is an analytic function in $\D$ with real coefficients.

 Now, suppose that $u$ and $v$ are harmonic functions in $\D$ with the same nodal set $Z$. Let further $z_0\in Z$ and $W$ be a neighborhood of $z_0$; suppose that $W$ admits a conformal mapping $\alpha$ into $\D$ such that $\alpha^{-1}(0)=z_0$ and $v\circ\alpha^{-1}(w)=\Im w^k$ for some $k$.  We consider analytic functions $U=\tilde u+iu$ and $V=\tilde v+iv$ on $W$ such that $\Im(U)=u$, $\Im(V)=v$ and $U(z_0)=V(z_0)=0$, then $V\circ \alpha^{-1}(w) = w^k $. Clearly, $u\circ\alpha^{-1}$ has the same zero set in $\D$ as $\Im w^k$, and by the argument above $U\circ \alpha^{-1}(w)=(g\circ V\circ \alpha^{-1})(w)$ for $w\in \D$, where $g$ is an analytic function with real coefficients. Thus $U(z)=g\circ V(z)$ in $W$. Our aim is to extend this statement to a larger class of pairs $U$ and $V$.

\begin{theorem}\label{th:struc}
Let $U$ and $V$  be analytic functions in the unit disc such that $Z(\Im U)=Z(\Im V)$. 
Assume also that   $\Omega=V^{-1}\{r_1<|z|<r_2\}$ is connected  for some $r_1<r_2$ and there exists  integer $k$ such that $V|_\Omega$ is a $k$-cover of $\{r_1<|z|<r_2\}$.  
Then  $U(z)=g\circ V(z)$ for $z\in\Omega$, where $g$ is an analytic function on $\{|z|<r_2\}$ with real coefficients. 
\end{theorem}

\begin{proof}
Let $S=\{z: r_1<|z|<r_2,-\pi<\arg(z)<\pi\}$, it is a simply connected open set and $V^{-1}(S)=\cup_{j=}^k D_j$ is a disjoint union of $k$  open subsets of $\Omega$. For each of them there is a covering $V|_{D_j}:D_j\rightarrow S$ that is a bijection. Thus we can find the inverse functions $V_1^{-1},..., V_k^{-1}$ that map $S$ onto $D_1,..., D_k$ respectively. 

Let $\gamma$ be a closed circle with radius $r\in(r_1,r_2)$ let $z_0\neq -r$ be a fixed point on $\gamma$. For each $D_j$ there is one point $p_j\in D_j$ such that $V(p_j)=z_0$ and a lift of $\gamma$ that starts at $p_j$ and ends at some $p_{j'}$.  Then $j\mapsto j'$ is a bijection and since $\Omega$ is connected this permutation has no cycles of length less then $k$. We renumerate the preimages of $S$ to make the bijection: $j\mapsto j+1$, $k\mapsto 1$.   

For each $j$ the function $U\circ V_j^{-1}$ is an analytic function on $S$ which takes real values on $(r_1,r_2)$. Therefore $U(V_j^{-1}(\overline{z}))=\overline{U(V_j^{-1}(z))}$ for any $z\in S$. Similarly, looking at the preimages of $\{z: r_1<|z|<r_2, 0<\arg(z)<2\pi\}$, we see that $U(V_j^{-1}(\overline{z}))=\overline{U(V_{j+1}^{-1}(z))}$. Then $U(V_j^{-1}(z))=U(V_{j'}^{-1}(z))$ and $U(V_j^{-1}(z))=U(V_{i}^{-1}(z))$ for any $i,j \in 1..k$. Thus if $V(z_1)=V(z_2)$ for $z_{1,2} \in \Omega$, then $U(z_1)=U(z_2)$.  That gives us $U=g\circ V$ on $V^{-1}(r_1<|z|<r_0)$, where $g$ is an analytic function on $B=\{r_1<|z|<r_2\}$, which takes real values on segments $\pm(r_1,r_2)$.   

Let $h$ be the harmonic continuation of $\Im g|_{\gamma}$ to the disc $\{|z|<r\}$. Then $\Im U = h \circ V$ on $\{|z|=r\}$. Since $\Im U$ and $h$ are harmonic functions in $\{|w|<r\}$ with the same boundary values, $\Im U$ and $h$ are equal on $\{|w|<r\}$. Then $g$ also admits analytic continuation ($-\tilde h + i h$) to $\{|w|<r\}$ such that $U=g\circ V$. Since $g$ takes real values on segments $\pm(r_1,r_2)$ it has real values on $(-r_2,r_2)$ and therefore $g$ has real coefficients.     
\end{proof}

\begin{corollary}\label{pr:struc}
Let $\{U_n\}$ and $\{V_n\}$ be sequences of analytic functions in $\D$ such that $Z(\Im U_n)=Z(\Im V_n)$, and  $V_n\rightrightarrows V=z^k$ in $\D$.  Then for any $\rho<1$ there exists $n_0=n_0(\rho)$ such that for $n>n_0$ we have $U_n(z)=g_n\circ V_n(z)$, when $|z|<\rho$, and $g_n$ is an analytic function with real coefficients. 
\end{corollary}

\section{Proof of Theorem \ref{main}}\label{s:main} \label{se:3}

\subsection{Harnack's inequality for the ratios}
Before we proceed to the proof of Theorem \ref{main}, we make some simple observations.
Suppose there exist two sequences $u_n$ and $v_n$  of harmonic functions in $\D$ such that $Z(u_n)=Z(v_n)$, the number of nodal domains of each $u_n$ and $v_n$ is not greater than $N$, and the ratios $f_n=u_n/v_n$ enjoy either 
\begin{equation} \label{infty}
{\rm{(a)}}\ \frac{\sup\limits_K|f_n| }{\inf\limits_K|f_n|} \mathop{\to}\limits_{n\to \infty} +\infty\quad\text{or}\quad
{\rm{(b)}}\ 
\frac{\sup\limits_K|\nabla f_n| }{\inf\limits_K|f_n|} \mathop{\to}\limits_{n\to \infty} +\infty.
\end{equation}

   By Lemma \ref{N2} we may assume $\alpha_n u_n$ normally converge to a non-zero harmonic in $\D$ function $u$,  and $\beta_n v_n$ normally converge to a non-zero harmonic in $\D$ function $v$, where $\alpha_n$ and $\beta_n$ are sequences of non-zero real numbers.  Multiplying $u_n$ and $v_n$ by constants, we do not change the properties (\ref{infty}a) and (\ref{infty}b), so we may assume that all $\alpha_n =1$ and all $\beta_n=1$. Then $u_n$ normally converge to $u$ and $v_n$ normally converge to $v$. The ratio $f_n= u_n /v_n$ does not vanish in $\D$ and we may also assume that all $f_n$ are positive in $\D$.

 Now we start the proof of (\ref{eq:H}) in Theorem \ref{main}.
 Assume the contrary. The observations from above reduce the question to the following statement.

\begin{prop}\label{pr:H}
 Let $\{u_n\}$ and $\{v_n\}$ be sequences of harmonic functions in $\D$ such that $Z(u_n)=Z(v_n), u_n=f_nv_n, f_n>0$ and $u_n\rightrightarrows u$, $v_n\rightrightarrows v$ in $\D$, where $u$ and $v$ are non-zero functions. Then \[\sup_n(\sup_K f_n/\inf_K f_n)<+\infty.\]  
\end{prop}
 
\begin{proof}
    Let $Z$ denote the nodal set of $u$ and $v$, see Lemma \ref{l:nodal}.  By $Z_s$ we denote the singular set of $Z$, namely  $Z_s=\{ z \in \D: u(z)= |\nabla u(z)|=0\}=\{ z \in \D: v(z)= |\nabla v(z)|=0\}$. (It is uniquely determined by $Z$, see Lemma 
\ref{lda}). The critical set is a countable set with no accumulation points within $\D$.
		
 We consider an open disc of radius $r\in(0,1)$ such that $K\subset \D_r \subset \D$ and $\partial \D_r \cap Z_s = \emptyset$. Note that $Z\cap \partial \D_r$ is the union of a finite number of points $z_i$, $1\le i\le k $. Each $z_i$ does not belong to the critical set, hence $|\nabla v(z_i)|> \varepsilon$  for some $\varepsilon>0$ and for any $i, 1\le i\le k$. We fix $i$ and consider a neighborhood $V_i \subset \D$ of $z_i$ such that $|\nabla v|(z)> \varepsilon/2 $ for $z \in  V_i$. Recall that $v_n$ normally converge to $v$ as $n \to +\infty$, it implies that  there exists a neighborhood $W_i\subset\overline{W_i}\subset V_i$ of $z_i$ such that $|\nabla v_n|(z)>\varepsilon/4$ for all $z \in W_i $ and for all $n$ large enough.
Further,  there exists  $M>0$ such that $ \sup\limits_{W_i}|v_n|< M$ and $\sup\limits_{ W_i}|u_n|< M$ for all $n \in \mathbb{N}$. 

The next step is to show that there exist a constant $C_i$ and a radius $r_i>0$ such that $\sup\limits_{B_i}|f_n|< C_i \inf\limits_{B_i}|f_n|$ for all $n \in \mathbb{N}$, where $B_i=B_{r_i}(z_i)$. If  $r_i$ is small enough, then $2\overline{B_i} \subset W_i$. By the standard Cauchy estimates there exist $A$, $R>0$ such that the following estimates of partial derivatives of $u_n$ and $v_n$ hold: 
 \[\sup\limits_{B_{r_i}(z_i)} |D^\alpha u_n| \leq AR^{|\alpha|} \alpha! \quad \& \quad \sup\limits_{B_{r_i}(z_i)} |D^\alpha v_n| \leq AR^{|\alpha|} \alpha!.
 \]
   The symbol $\alpha$ denotes the multi-index, $A$ and $R$ do not depend on $n$.
	
We know that $u(z_i)=v(z_i)=0$, it implies that for any $d>0$ there are $\xi_+$ and $\xi_-$: $v(\xi_-)<0$ and $v(\xi_+)>0$,
 $|z_i-\xi_+|<d$, $|z_i-\xi_-|<d$, hence for any $n$ large enough $v_n(\xi_-)<0$ and $v_n(\xi_+)>0$, so there is a zero of $v_n$ in a segment $[\xi_-,\xi_+]$. The last argument implies that there exists a sequence of points $\xi_n$ such that  $\xi_n \to z_i$ and $v_n(\xi_n)=0$.
Now we apply Lemma \ref{lda} to $D_i=B_{3r_i/2}$. Then \eqref{eq:tay} implies
\[
\sup\limits_{B_{r_i}(z_i)}|f_n| \leq c_i\inf\limits_{B_{r_i}(z_i)}|f_n|.
\]
   
 Let $z \in (\partial \D_r\setminus Z)$, then there is $r_z>0$ such that $u$ and $v$ do not vanish in $B_{2r_z}(z)$. Then for $n$ large enough $u_n$ and $v_n$ do not vanish in $B_{3r_z/2}(z)$ and by the classical  Harnack inequality there is $c_z>0$ such that  
\[\sup\limits_{B_{r_z}(z)}|f_n| \leq c_z\inf\limits_{B_{r_z}(z)}|f_n|,\quad  {\text{for all} }\quad n>n_0(z).\] Note that $\{B_r(z) \}_{z\in \partial \D_r} $ form an open covering of the compact set $\partial \D_r$, using the standard compactness argument, it is easy to see that 
\[ 
\sup\limits_{ \partial \D_r}|f_n| \leq C \inf\limits_{\partial \D_r}|f_n|.
\]
The proposition follows from the maximum and minimum principles for the ratios of harmonic functions, see \cite{LM}.
\end{proof}

\subsection{Gradient estimate for the ratios} 
We assume the contrary, as above,  and take sequences of functions $u_n$ and $v_n$ such that (\ref{infty}b) holds and $u_n, v_n$ converge normally to $u$ and $v$ respectively, once again $Z(u)=Z(v)=Z$.



Now, if (\ref{infty}b) holds then, by Proposition \ref{pr:H}, we have
\[\sup_n\sup_{x\in K}\frac{|\nabla f_n(x)|}{|f_n(x)|}=+\infty.\]
We may assume that 
$|\nabla f_n(x_n)|/f_n(x_n)\rightarrow\infty$ and $x_n$ converge to $x_0$.

For each $a\in K\cap Z$, let $B(a)$ be a disc  with center at  $a$ that admits a conformal mapping $\beta_a:B(a)\to\D$  such that $\beta(a)=0$ and $v\circ \beta_a^{-1}(w)=\Im w^k$ for some $k=k(a)\ge 0$, $\beta(B(a))=\D$. Further, let $D(a)=\beta_a^{-1}(\frac14\D)$. Then $K\cap Z$ can be covered by finitely many  of the sets $\{D(a)\}_{a\in K\cap Z}$, let $K\cap Z\subset \cup_{j=1}^J D(a_j)=O$ and $\delta_1=\dist(K\setminus O,Z)$. 
 
Then either  $x_0\in D(a_j)$ for some $j$ or $B_{\delta_1}(x_0)$ does not intersect $Z$.  In the latter case $u_n$ and $v_n$ do not change sign in $B_{\delta_1/2}(x_0)$ for $n$ large enough and the usual Harnack inequality for positive harmonic functions leads to a contradiction.  Otherwise we write $\beta=\beta_{a_j}$ and  define $\tilde{g}=g\circ\beta^{-1}$, where $g\in\{u_n,v_n,u,v,f_n\}$.
 Clearly $|\nabla \beta^{-1}|$ is bounded in $\frac14\D$. We have reduced the gradient estimate for the ratios  to the following

\begin{prop} Let $\{\tilde{u}_n\}$ and $\{\tilde{v}_n\}$ be sequences of harmonic functions in $\D$ such that $Z(\tilde{u}_n)=Z(\tilde{v}_n), \tu_n=\tf_n\tv_n, \tf_n>0$ and $\tu_n\rightrightarrows \tu$, $\tv_n\rightrightarrows \tv$ in $\D$, where $\tv=\Im z^k$. Then 
\[
 \sup_n\sup_{\frac14\D}\frac{|\nabla \tf_n|}{\tf_n}<+\infty.\]
\end{prop}

\begin{proof} Let $\tV_n$ be analytic in $\D$ with $\Im \tV_n=\tv_n$ and such that $\tV_n\rightrightarrows z^k$, $\tV_n=\tw_n+i\tv_n$. By Corollary \ref{pr:struc}  for each $r_0<1$ we have $\tu_n=\Im (g_n\circ\tV_n)$ in $r_0\D$ for all $n=n(r_0)$ large enough, where $g_n=\sum_1^\infty a_{n,j}z^j$ is an analytic function in $r_1\D$ with real coefficients $a_{n,j}$, $r_1<r_0^{1/k}$. We get
\begin{multline}\label{eq:ratio}
\tf_n=\frac{\tu_n}{\tv_n}=\sum_{j=1}^\infty a_j\frac{\Im(\tw_n+i\tv_n)^j}{\tv_n}=\\
\sum_{j=1}^\infty a_j\sum_{k=0}^{[(j-1)/2]}(-1)^k{j\choose 2k+1}\tw_n^{j-2k-1}\tv_n^{2k}.
\end{multline}
By Proposition \ref{pr:H}, $\tf_n$ are bounded from above and below in $\frac14\D$ uniformly in $n$ (since $\tf_n(x_0)\rightarrow\tu(x_0)/\tv(x_0)$ when $x_0\not\in Z(v)$). Then it is enough to show that $|\nabla f_n|$ are bounded from above in $\frac14\D$. 

We have  $|\tv_n|,|\tw_n|\le M_0$ and $|\nabla \tv_n|=|\nabla \tw_n|\le M_1$ in $\frac 14 \D$, for some constants $M_0$ and $M_1$ and $n\ge n_0$,  we may also assume that $M_0<1/3$ by taking $n_0$ large enough. Further, since $\tilde{u}_n\rightrightarrows \tilde{u}$  in $\D$, we get that $g_n$ are uniformly bounded in $r_1\D$, $|g_n|\le A$ in $r_1\D$. Then, the Cauchy estimate implies $|a_j|\le Ar_1^{-j}$. Finally,
\[
|\nabla \tf_n|\le AM_1\sum_{j=1}^\infty j2^jM_0^{j-2}r_1^{-j}<+\infty,\]
when $r_0$ and $r_1$ are chosen close to $1$.
\end{proof}
\subsection{Concluding remarks} 
Theorem \ref{main} implies the following corollary, which generalizes the standard Cauchy estimate.
 \begin{corollary} \label{cor:m}
 Let $u$ and $v$ be harmonic functions in the disc $r\D \subset \mathbb{R}^2$ of radius $r$ such that $Z(u)=Z(v)$ and let  $f$ be the ratio of $u$ and $v$.  Suppose the number of nodal domains of $u$ (and $v$) is less than a fixed number $N$. Then there exists $C$, depending  on $N$ only such that 
$|\nabla \log|f||(0)  \leq Cr^{-1}$ . 
\end{corollary} 

We have obtained estimates for the ratios of harmonic functions and their gradients. Following the same pattern and using the expression for the ratio as in \eqref{eq:ratio}, we can also show that 
$\max_K|D^\alpha f|\le C\alpha! R^{|\alpha|}$
where $f=u/v$ and $C=C(K;N)$. See \cite{LM} for a similar argument.

\begin{remark} \label{r:length} Suppose that the length of $Z$ is bounded by $L$.  Then the uniform estimates \eqref{eq:H}, \eqref{eq:GE} will remain true with the constants $C_{1,2}$ depending on $L$ and $K$ only. It can be explained by the  fact that the number  of nodal domains in the unit disc can be estimated from above  by the length of the nodal set in a bigger disc and vice versa. We provide some references here. First, the number of nodal domains can be estimated by the doubling constant (the definition is given in Section \ref{se:4} below) in a larger disc, and vice versa, see \cite{N98}.
The estimate of the length of the nodal set from above by the 
doubling constant is well known, see for example \cite{H07}. The reverse 
estimate follows from the connection between the number of sign changes on a circle
and the doubling constant, which can be found in \cite{NPS}, see also 
references therein.

\end{remark}
Theorem \ref{th:struc}  on analytic functions taking real values on the same curves  can be considered in a more general and complicated framework of structure theory by K.~Stephenson, \cite{S86}. We gave an elementary proof for the case needed in this note.
  
All our arguments, except for the proof of Lemma \ref{l:nodal}, were essentially two-dimensional. We don't know if, for example, Proposition 
\ref{pr:H} holds in higher dimensions. 


\section{Toolbox for higher dimensions} \label{se:4}
\subsection{Classical Harnack inequality and elliptic estimates}
 The following facts about harmonic functions in the unit ball of $\R^n$ are well-known and follow immediately from the Poisson formula. Let $0<r<1$, there exist constants $h_r, a_r, b_r$ that depend on $r$ and $n$ only such that 
\begin{itemize}
\item {\it Harnack's inequality:} for any positive harmonic function $u$ in the unit ball 
\[\inf_{B_r}u\ge h_r\sup_{B_r}u.\]
\item {\it Cauchy estimates:}  for each multiindex  $\alpha$ and any harmonic function $u$ one has
\[
\sup_{B_r}|D^\alpha u|\le \alpha! a_r^{|\alpha|}\sup_{B_1}|u|.\] 
\item{\it Equivalence of norms:} for any harmonic function $u$ in the unit ball
\[\sup_{B_r}|u|\le b_r\left(\dashint_{\partial B_1}|u|^2\right)^{1/2}.\] 
\end{itemize} 
\subsection{Doubling constants} Let $u$ be a non-zero harmonic function in some domain $\Omega$. For each $x\in\Omega$ and $\rho<\dist(x,\partial\Omega)$ let 
\[
H_u(x,\rho)=\dashint_{\partial B_{\rho}(x)}|u|^2, \quad N_u(x,\rho)=H_u(x,2\rho)H_u(x,\rho)^{-1}.\]
 By $\dashint_{\partial B_{\rho}(x)}$ we denote the integral with respect to the normalized surface measure on $\partial B_{\rho}$ such that $\dashint_{\partial B_{\rho}(x) } 1 =1 $.

By considering the expansion of $u$ in homogeneous harmonic polynomials, it is not difficult to see that $\log H_u(x,\rho)$ is a convex function of $\log \rho$ and therefore $N_u(x,\rho)$ is a non-decreasing function of $\rho$. See \cite{H07} for the details.
 Further, $\lim_{\rho\to 0} N_u(x,\rho)=2^{2k},$ where $k$ is the order of vanishing 
of $u$ at $x$.

Given a harmonic function $u$ with $Z(u)\neq \emptyset$, we define \[\delta_u(x)=\dist(x,Z(u)),\] clearly $\delta_u$ depends on $Z=Z(u)$ only. We will skip sub-index $u$ in $\delta_u$, when it does not lead to any ambiguity.
\begin{lemma} \label{AL4}
 There exists a  constant $K>1$ depending on the dimension $n$ only such that for any function $u$ harmonic in $B_1$ with $Z(u)\cap B_{1/2}\neq\emptyset$ and any $x \in B_{1/4} $ with $u(x) \neq 0$ and $\delta_u(x)< (4K)^{-1}$ there exists a point $\tilde x $ for which  $|\tilde x- x|\leq K \delta_u(x)$ and $|u(\tilde x)| \geq 2|u(x)|$.
\end{lemma}
\begin{proof}
 Let $y\in Z$ be a such point that $|y-x|=\delta(x)$. Then $u(y)=0$ and  $\lim\limits_{\rho \to 0} N_u(y,\rho) \geq 4$. Therefore $N_u(y,\rho) \geq 4$ for any $\rho \in (0, 1/2)$ since $N_u$ is non-decreasing in $\rho$. 

Assuming that $K>2^s+1$, where $s$ is a positive integer, we get
\begin{align}
\max\limits_{B_{(K-1)\delta(x)}(y)} u^2 \geq \dashint_{B_{(K-1)\delta(x)}(y)} u^2 
\ge 4^{s-1} \dashint_{B_{2\delta(x)}(y)} u^2.
\end{align}
  By the equivalence of norms, we have 
$$\max\limits_{B_{(K-1)\delta(x)}(y)} u^2 \geq b_{1/2}^{-2} 4^{s-1}|u(x)|^2\ge 4|u(x)|^2, $$ 
for $K$ (and $s$) large enough. Hence there exists $\tilde x\in B_{(K-1)\delta(x)}(y)$ such that $|u(\tilde x)|\geq 2 |u(x)|$. Clearly,  $|x-\tilde x| \le |x-y|+|y-\tilde x|\leq K \delta(x)$.
\end{proof}
 We remark that the sign of $u(\tilde x)$ can be opposite to the sign of $u(x)$.

Let $v$ be a given non-constant harmonic function in $B_1$ and let $A$ be the maximum of $|v|$ over $B_{1/4}$. Define
 \[m= \min\limits_{a\in [-A,A], x \in \overline{B}_{1/4}} \dashint_{\partial B_{1/4}(x)} (v-a)^2,\quad M= \max\limits_{a\in [-A,A], x \in \overline{B}_{1/4}} {\dashint_{\partial B_{1/2}(x)} (v-a)^2 }.\]  Since $v$ is a non-constant harmonic function,
 $ m $ is greater than $0$.
 Also note that
 $M < +\infty $.  
Then for any $x\in B_{1/4} $  and $r<1/4$ we get 
\begin{equation} 
N_{v-v(x)}(x,r)\le \frac{\dashint_{\partial B_{1/2}(x)} (v-v(x))^2}{\dashint_{\partial B_{1/4}(x)} (v-v(x))^2}   \leq M/m.
\end{equation}
We call $N_1(v):=\max_{x\in B_{1/4}, r\in(0,1/4)}N_{v-v(x)}(x,r)$ the generalized doubling constant of $v$.

\subsection{\L ojasiewicz exponents}
 The following well-known fact is related to general real analytic functions. 
 For any function $f$, real analytic in $B_1$, with $Z(f) \neq \emptyset$ there exist constants $l, L, \gamma >0$ depending on $f$ such that $$ L \cdot d(x, Z(f)) \geq |f(x)| \geq l \cdot d(x, Z(f))^\gamma$$ for any $x \in B_{1/2}$; we refer the reader to the textbook \cite{KP} or to the original work of S.~\L ojasiewicz \cite{L}.

Later we will apply this fact to a fixed harmonic function with a prescribed zero set. It will be convenient to measure the distance from a point to the zero set by evaluating the harmonic function at this point.

\section{Key estimate}  \label{se:5}
\subsection{Main proposition}
   We fix the nodal set $Z\subset B$ and one harmonic function $v$ such that $Z(v)=Z$, as before $\delta(x)=\delta_v(x)=\dist(x,Z)$. We will assume that $Z\cap B_{1/2}\neq\emptyset$, otherwise the statement of Theorem \ref{H} follows from the classical Harnack inequality.  The aim of this section is the prove the following statement.
 \begin{prop} \label{ML}
   There exist  constants $M=M(Z)>0$ and $c=c(Z)>0$ such that $$\sup\limits_{B_{1/16}} |u| \leq M \sup\limits_{y\in B_{1/2}, \delta(y)\geq c} |u|(y)$$ for any function $u \in H_Z$.
\end{prop}
 The constants in the Proposition and in Lemmas below depend  on the dimension and on $v$ (or, equivalently, on $Z$) only, unless otherwise stated. They can be expressed explicitly through  constants depending only on the dimension, the generalized doubling constant $N_1(v)$ and constants in the 
\L ojasiewicz inequalities for $v$.

\subsection{Three lemmas}  We postpone the proof of the proposition and start with auxiliary lemmas.

 \begin{lemma}\label{AL1}
   There exists a constant $C=C(v)>1$ such that for any $x \in B_{1/4}$ with $v(x) \neq 0$ there is $\tilde x \in B_{1/2}$ with $|x-\tilde x| \leq \frac{3}{4}\delta(x)$ and $|v(\tilde x)| \geq C |v(x)|$. 
\end{lemma}

The statement looks similar to Lemma  \ref{AL4} but has a very different nature. Here we find the new point $\tilde x$ in the same nodal domain as $x$, the constant $K$ from Lemma \ref{AL4} turns into $3/4$ but now the constant $C$ (that was equal to two in Lemma \ref{AL4}) depends on the function $v$.

\begin{proof}
Consider any $x \in B_{1/4}$.  If  $\delta(x) \ge 1/8$, put $\tilde x$ to be the point on $\partial B_{1/16}(x)$ at which $|v|$ attains the maximal value. Clearly, 
\[C_0(v)=\inf_{x\in B_{1/4}}\frac{\max_{\partial B_{1/16}(x)}|v|}{|v(x)|}>1.\] 
If $\delta(x)<1/8$, then
\begin{multline*}\frac{\max_{\partial B_{\delta(x)}(x)} |v-v(x)|^2}{\max_{\partial B_{\delta(x)/2}(x)} |v-v(x)|^2} \leq c\frac{\dashint_{\partial B_{2\delta(x)}(x)} (v-v(x))^2}{\dashint_{\partial B_{\delta(x)/2}(x)} (v-v(x))^2}\\ \leq c\frac{\dashint_{\partial B_{1/2}(x)} (v-v(x))^2}{\dashint_{\partial B_{1/8}(x)} (v-v(x))^2}\le cN_1(v)^2, \end{multline*}
where $c=b_{1/2}^2$ from the elliptic estimate.
The last two inequalities follow from monotonicity of the doubling constant. 
 That implies $$\max\limits_{\partial B_{\delta(x)}(x)} |v-v(x)|^2 \leq cN_1(v)^2\max\limits_{\partial B_{\delta(x)/2}(x)} |v-v(x)|^2.$$ 
We may assume that $v(x)>0$, then we have
$$ \max_{\partial B_{\delta(x)/2}(x)} |v-v(x)| \geq c_1 {\max_{\partial B_{\delta(x)}(x)} |v-v(x)|} \geq c_1  v(x),$$ 
where $c_1=c^{-1/2}N_1(v)^{-1}$.

 Denote $\max\limits_{\partial B_{\frac{3}{4}\delta(x)}(x)}(v - v(x))$ by $A=A(x)$. Let us show that \[A \geq c_2 \max\limits_{\partial B_{\delta(x)/2}(x)} |v-v(x)|.\] 
Clearly, by the maximum principle, $A\ge \max\limits_{\partial B_{\frac{1}{2}\delta(x)}(x)}(v - v(x))>0$. The function $\tilde v(\cdot):= A - v(\cdot)+v(x)$ is positive on  $B_{\frac{3}{4}\delta(x)}(x)$ 
and therefore, by the Harnack  inequality,  \[A= \tilde v(x) \geq  h_{2/3} {\max\limits_{\partial B_{\delta(x)/2}(x)}  (A -v+v(x))} \geq h_{2/3}{\max\limits_{\partial B_{\delta(x)/2}(x)}  ( -v+v(x))},\] where $0<h_{2/3}<1$ is a constant depending on the dimension only.
 We conclude 
\[ \max\limits_{\partial B_{\frac{3}{4}\delta(x)}(x)} (v - v(x)) =A\geq h_{2/3} \max\limits_{\partial B_{\delta(x)/2}(x)} |v-v(x)| \geq h_{2/3} c_1 v(x).\]  Thus ${\max\limits_{\partial B_{\frac{3}{4}\delta(x)}(x)} v \geq C v(x)}$ with $C= 1 +h_{2/3} c_1$.
\end{proof} 

  Now, assume that $x\in B_{1/8}$ and $v(x)\neq 0$. 
 Then, applying Lemma \ref{AL1} several times, we can construct a finite sequence  $x_0=x, x_1, \dots, x_m$ such that for $i=0,1,...,m-1$
\begin{enumerate}
\item[(i)] $|v( x_{i+1})| \geq C |v(x_i)|$, 
\item[(ii)] $|x_{i+1}-x_i| \leq \frac{3}{4}\delta(x_i)$, 
\item[(iii)] $x_i\in B_{1/4}$ and $x_m \in B_{1/2}\setminus B_{1/4}$.
\end{enumerate}
\begin{lemma}\label{AL2}
 There exists $c=c(v)\in(0,1/2)$ such that for any $x\in B_{1/8}$ and $x_m$ defined above the inequality $\delta(x_m)\geq c(v)$ holds.
\end{lemma}
\begin{proof}
  First, recall that  there exist positive constants $L$, $l$ and $\gamma\geq 1$ depending on $v$ such that 
 \begin{equation} \label{LE}
L \delta(x) \geq |v(x)| \geq l  \delta^\gamma(x)
\end{equation}
 for any $x \in B_{3/4}$. 
 It is sufficient to show that $|v(x_m)| \geq c(v)$ for some constant $c(v)>0$.

 By (i) we have $|v(x_i)|\leq C^{i-m} |v(x_m)|$ for $i \in \{0, 1, \dots , m\}$.
Then $$\delta(x_i) \leq {l^{-1/\gamma}}|v(x_i)|^{1/\gamma} \leq {l^{-1/\gamma} C^{\frac{i-m}{\gamma}} |v(x_m)|^{1/\gamma}}.$$ 
Since $C>1$, the sum  $S:=\sum_{j=0}^{\infty} C^{\frac{-j}{\gamma}} $ is finite and, by (ii), $$|x_0-x_m|\leq\sum_{i=0}^{m-1}|x_i-x_{i+1}| \leq \sum_{i=0}^{m-1}\delta(x_i)\leq S l^{-1/\gamma}|v(x_m)|^{{1/\gamma}}.$$
 Recall that $x_0 \in B_{1/8}$ and $ x_m \notin B_{1/4}$. It implies $ |x_0-x_m| \geq 1/8$. Thus $ |v(x_m)|\ge c_3$, where $c_3=c_3(v)$.  
\end{proof}
 \begin{lemma} \label{AL3}
   There exists a constant $\beta=\beta(Z)=\beta(v)>0$  such that for any function $u \in H_Z$ and a point $x\in B_{1/8}$ there is a point $y\in B_{1/2}$ such that $\delta(y)\geq c(v)>0$ and $|u(y)| \geq |u(x)| \delta^{\beta}(x) $, where $c(v)$ was defined in Lemma \ref{AL2}.
\end{lemma}
\begin{proof} 
 We will assume that $\delta(x) \leq 1/2$, otherwise we can take $x=y$.
 Applying the construction from Lemma \ref{AL2} (for the point $x$ and the function $v$), we  put $y=x_m$, then $\delta(y)\geq c>0$ and $y \in B_{1/2}$. 

Recall that, by (ii), $|x_{i+1}-x_{i}|\leq \frac{3}{4} \delta(x_i)$. Since $u$ does not change sign in $B_{\delta(x)}(x)$,  the Harnack inequality implies $|u(x_i)|\leq h_{3/4}^{-1} |u(x_{i-1})|$, where $h_{3/4}$ depends on the dimension only. Therefore  $|u(x)|=|u(x_0)| \leq h_{2/3}^{-m} |u(y)|$.

 By (i), we have $|v(x_m)| \geq C^m |v(x_0)|$ and, by \eqref{LE}, 
\[l \delta^{\gamma}(x_0) \leq|v(x_0)|,\quad |v(x_m)|\leq L \delta(x_m)\le L .\] It shows that $ \delta^{\gamma}(x_0) \leq Ll^{-1}C^{-m}$. We can choose $\beta_1=\beta_1(v)>0$ such that $C^{-\beta_1}\le h_{2/3}$ and then find $\beta_2=\beta_2(v)>0$ for which  $2^{\beta_2} \geq (Ll^{-1})^{\beta_1}$. Since $\delta(x)=\delta(x_0)\leq 1/2$, we  obtain 
\[\delta^{\gamma\beta_1+\beta_2}(x)\leq (Ll^{-1}C^{-m})^{\beta_1}2^{-\beta_2}\le h_{2/3}^{m} .\] Finally, when  $\beta=\gamma\beta_1+\beta_2$, we obtain the required inequality.
\end{proof}
\subsection{Proof of Proposition \ref{ML}}
 Suppose that $u\in H_Z$ and let $c=c(v)$ be as in  Lemma \ref{AL2}, we will prove Proposition \ref{ML} with this $c$. Suppose that  $\sup\{ |u(y)|: y\in B_{1/2}, \delta(y)\geq c\} \leq 1$. We wish to prove that \[\sup \{|u(y)|: y\in B_{1/16}\}\leq M\] for some $M=M(c,Z)>0$. 

Suppose that $x_0 \in B_{1/16}$ and $|u(x_0)|> M_0$. Applying Lemma \ref{AL4} for $x_0$, we can find $x_1$ such that $d( x_1, x_0) \leq K \delta(x_0)$ and $|u(x_1)| \geq 2|u(x_0)|$. Let us consequently employ Lemma \ref{AL4} infinitely many times and find the sequence $\{x_i\}_{i=1}^{\infty}$ with $|u(x_{i+1})| \geq 2|u(x_i)|$ and $d( x_{i+1}, x_i) \leq K \delta(x_i)$. However we may use the lemma for $x_i$ only if $x_i \in B_{1/4}$. Let us show that all $x_i$ are in $B_{1/8}$ if $M_0$  is large enough.

 By Lemma \ref{AL3} if $x_i \in B_{1/8}$, then there is $y_i$ such that $y_i\in B_{1/2}, \delta(y_i)\geq c$ and $|u(x_i)|\leq |u(y_i)| (1/\delta(x_i))^{\beta} \leq (1/\delta(x_i))^{\beta}$. Note that \[|u(x_i)| \geq 2^i |u(x_0)| \geq 2^i M_0.\] We conclude  
\begin{eqnarray*}
\delta(x_i) \leq \frac{1}{|u(x_i)|^{1/\beta}} \leq M_0^{-1/\beta} 2^{-i/\beta},\\
 d( x_{i+1}, x_i) \leq K \delta(x_i) \leq K M_0^{-1/\beta} 2^{-i/\beta}.
\end{eqnarray*}
 It is easy to see that the sum  $\sum\limits_{i=0}^{+\infty} K M_0^{-1/\beta} 2^{-i/\beta}$ is finite. If $M_0$ is sufficiently large, then  $\sum_{i=0}^{\infty}d( x_{i+1}, x_i) < 1/16$. Since $x_0\in B_{1/16}$, all $x_i$ are in $B_{1/8}$ and so does the limit $\lim\limits_{i\to+\infty}x_i=:x_\infty$. Here comes the contradiction with $\lim\limits_{i\to+\infty}|u(x_i)|=+\infty$.   

\section{The Harnack inequality for the ratios} \label{se:6}
 
\subsection{Proof of Theorem \ref{H}} 
 We cover $B_{1/2}$ by a finite number of balls $B_j$ with centers in $B_{1/2}$ and radii $1/32$. For each $B_j$ we denote by $D_j$ the concentric ball of radius $1/2$. Now, we apply Proposition \ref{ML} to the function $v$ and each of the finitely many ball $D_j$ (in place of the unit ball). Taking the maximum of the corresponding constants $M$ and the minimum of the constants $c$, we obtain
$$\sup\limits_{B_{1/2}}|u| \leq M \sup \limits_{y\in B_{3/4},\delta(y)\geq c } |u|(y)  $$ for 
any $u\in H_Z$ and some $c,M$ depending on $Z$ only.

 First, we wish to show that for any $y_0 \in B_{1/2}\setminus Z$ there exists $C_1=C_1(y_0,Z)$ such that $\sup\limits_{B_{1/2}}|u| \leq C_1 |u(y_0)|$ for any $u\in H_Z$. It suffices to establish 
\begin{equation} \label{eq:13}
 \sup\limits_{y\in B_{3/4},\delta(y)\geq c }|u| \leq C_2 |u(y_0)|.
\end{equation}
Let $\Omega_i$, $i=1..k$ be the connected components of  $B_{3/4} \setminus Z$. Put 
\[V_i:= \closure{\Omega_i\cap \{y\in B_{3/4}: \delta(y)\geq c) \}}.\] 
Clearly, $V_i$ is a compact subset of a nodal domain of $u$. Decreasing $c$, if necessary, we may assume that each $V_i$ is non-empty. Fix any points $y_i \in V_i$. By the Harnack inequality, $\sup\limits_{V_i}|u| \leq C_{3,i} |u(y_i)|$. Hence 
\[\sup_{B_{1/2}}|u|\le C\sup\limits_{y\in B_{3/4},\delta(y)\geq c }|u|\leq C_4 \max\{|u(y_1)|,\dots, |u(y_n)|\}.\]   
To establish \eqref{eq:13} we will show that if $\max\{|u(y_1)|,\dots, |u(y_n)|\}=1$, then $|u|(y_0) \geq c_1$ for some $c_1=c_1(Z)>0$. Assume the contrary, suppose there is a sequence of functions $u_i \in H_Z $ such that $\max\{|u_i(y_1)|,\dots, |u_i(y_n)|\}=1$ and $|u_i(y_0)| \to 0$ as $i \to \infty$. 
Since $\sup\limits_{B_{1/2}}|u| \leq C_4$, we can choose a subsequence of $u_i$ to be uniformly converging on compact subsets of $B_{1/2}$. Let $u$ be the pointwise limit of such subsequence in $B_{1/2}$. Then $u$ is a harmonic function in $B_{1/2}$. 
 However $|u_i(y_0)| \to 0$, hence $u(y_0)=0$. Let $y_0$ lie inside of the nodal domain $\Omega_j$. We may assume that all $u_j$ are positive in $\Omega_j$, then the pointwise limit $u$ is non-negative in $\Omega_j$ and $u$ is equal to zero in the interior point $y_0$ of $\Omega_j$. Thus by the strict  maximum principle $u$ is identically zero. Since $\max\{|u_i(y_1)|,\dots, |u_i(y_n)|\}=1$, then $\max\{|u(y_1)|,\dots, |u(y_n)|\}=1$ and $u$ is not identically zero.

 By the contradiction above we have obtained 
\begin{equation} \label{eq:N1}
  |u(y_0)| \leq \sup\limits_{B_{1/2}}|u| \leq C_1(y_0,Z) |u(y_0)|
\end{equation}
 for any $u\in H_Z$.

 Now, let $x_0 \in Z\cap B_{1/2}$. Assume that the homogeneous polynomial expansion of $u$ at $x_0$ starts with a non-zero homogeneous polynomial of order $k$: $u(x)=\sum_{i=k}^{+\infty} p_{i,u}(x-x_0).$  By the local division principle $p_{k,u}=c_u p$, where $p=p(Z,x_0)$. Now, we fix $p$ and wish to show that 
\begin{equation}\label{eq:N2}
  |c_u| \leq C_5(x_0, Z) \sup\limits_{B_{1/2}}|u| \leq C_6(x_0,Z) |c_u|
\end{equation}
 for any $u\in H_Z$. 

 The first inequality is trivial and  follows immediately from the standard Cauchy estimates of derivatives of harmonic functions.  
 The proof of the second inequality is similar to the proof of \eqref{eq:N1}. Assume the contrary, suppose there exist $u_i \in H_Z$, $c_{u_i} \to 0$ as $i\to +\infty$, but $\sup\limits_{B_{1/2}}|u_i|=1$. We may assume that all $u_i$ have the same sign in each component of $B_1\setminus Z$. By \eqref{eq:N1}, we can choose a subsequence of $u_i$ which  normally converges in $B_{1/2}$ to a non-zero harmonic function $u$. Moreover the nodal set of $u$ in $B_{1/2}$ will be $Z\cap B_{1/2}$. The order of vanishing of $u$ at $x_0$ must be $k$ as well. But the normal convergence implies $c_u=0$ and the contradiction is found.

Now, fix a point $y_0 \in B_{1/2}\setminus Z$ and consider functions $\tilde{u},\tilde{v} \in H_Z$ with $\tilde{u}(y_0)=\tilde{v}(y_0)=1$. Let $y$ be an arbitrary point in $B_1$.
 By \eqref{eq:N1} we know that $\sup_{B_1}|\tilde{u}|$ and $\sup_{B_1}|\tilde{v}|$ are not greater than $C_1$. Hence, by the standard Cauchy estimates, we obtain 
$|D^\alpha \tilde{u}(y)|\leq a r^{|\alpha|} \alpha!$ and  $|D^\alpha \tilde{v}(y)|\leq a r^{|\alpha|} \alpha !$ for any multi-index $\alpha$, where $a,r>0$ and depend only on $C_1$ and $y$. Further, the first homogeneous polynomial in the Taylor expansion of $\tilde{v}$ at $y$ is equal to $c_{\tilde v} p$, where $p=p(Z,y)$and the coefficient $|c_{\tilde v}|>c(Z,y)>0$ by \eqref{eq:N2}. Then, applying \eqref{eq:tay} in Lemma \ref{lda} to $f=\frac{\tilde u}{\tilde v}$, we obtain that $|D^\alpha f(y)|\ \leq A_y R_y^{|\alpha|} \alpha !$, where $A_y,R_y$ depend  on $a,r,c(Z,y), p(Z,y)$ only.
 
 Note, that the constants $A_y, R_y$ depend on $y$.
 However, using the real analyticity of $f$, we may conclude that  $|D^\alpha f(x)|\ \leq 2A_y (2R_y)^{|\alpha|} \alpha !$ for any  $x\in B_{\varepsilon}(y)$, where $\varepsilon=\varepsilon(A_y, R_y)$ .   
Further, we may cover $\overline{B}_{1/2}$ by $\bigcup\limits_{y \in \overline{B}_{1/2}} B_{\varepsilon(y)}(y)$ and choose a finite covering $\overline{B}_{1/2} \subset \bigcup_{i=1}^{m} B_{\varepsilon(y_i)}(y_i)$. In each $B_{\varepsilon(y_i)}(y_i)$ we find corresponding $A_i, R_i$ and put $A=\max(A_1, \dots, A_m)$ and $R=\max (R_1, \dots, R_m)$.

 Finally, $|f(y)|\leq A$ for any $y\in B_{1/2}$. If we swap $\tilde u$ and $\tilde v$, we obtain $1/|f(y)|\leq A$.  That gives us the Harnack inequality $\sup\limits_{B_{1/2}}|f| \leq A^2\inf\limits_{B_{1/2}}|f|$  and the gradient estimate $\sup\limits_{B_{1/2}}|D^\alpha f| \leq A^2 R^{|\alpha|}\inf\limits_{B_{1/2}}|f|$.

\subsection{Concluding remarks and questions.}
A very natural question is how one can find (non-trivial) pairs of real-valued harmonic functions with the common zero set. In dimension two the situation is fairly well understood, due to the connections with complex analysis; we refer the reader to \cite{S86, M} for examples and further discussion. In higher dimensions simple examples can be constructed by extending functions of two variables or by applying the Cauchy-Kovalevskaya theorem, see \cite{LM} for details, but it is not clear how to describe all pairs of harmonic functions that share the same zero set. Related questions on hypersurfaces where families of eigenfunctions vanish were recently discussed by J. Bourgain and Z. Rudnick in \cite{BR} and by M. Agranovsky in \cite{A}. A non-trivial example of an infinite family of  harmonic polynomials in dimension four (and some higher dimensions) that vanish on the same set in the unit ball was given in \cite{LM}, by constructing a homogeneous harmonic polynomial of degree two that divides infinitely many linearly independent harmonic polynomials.  To the best of our knowledge, the question if such non-trivial families exist in dimension three is open.

Another question about entire real valued harmonic functions was raised by D. A.  Brannan, W. H. J. Fuchs,  W. K. Hayman  and
             {\"U}.  Kuran in \cite{BFHK}. It is known that every three entire harmonic functions in $\mathbb{R}^2$ that have the same zero set are linearly dependent. In dimension $3$ the last claim is not true, see example in \cite{BFHK}.
 Suppose that $u$ is a harmonic function in $\mathbb{R}^n$ and $\log|u(x)| \leq o(|x|)$ as $|x| \to \infty$. Is it true that any harmonic function in $\mathbb{R}^n$ with the same nodal set as $u$ is a multiple of $u$?
 For instance, it is true and known if $u$ is a homogeneous harmonic polynomial. One can use Theorem \ref{H} or some other way to see that.

 Given an entire harmonic function  in $\mathbb{R}^n$, one can consider its analytic extension to $\mathbb{C}^n$. Theorem \ref{H} shows that if two harmonic functions $u$, $v$ in the unit ball $B\subset \mathbb{R}^n$ have the same zero set $Z$, then their complex zeros coincide in some complex neighborhood of $B$.   Is it true that the zeros in $\mathbb{R}^n$  of a real valued entire harmonic function $u$  uniquely determine its complex zeros in $\mathbb{C}^n$ if $u$ is of exponential type zero? It is not true without assumption on exponential type zero. For instance, $e^x \sin y$ and  $\cosh x \sin y $   have the same real zeros but not the complex zeros. 
 The positive answer to the question from \cite{BFHK} that was formulated above would surely imply the positive answer to the question about complex and real zeros.

\subsection*{Acknowledgments} 

We would like to thank Dan Mangoubi for interesting discussions of the problem.  We are also grateful to Dmitry Khavinson for motivating questions on complex and real zeros of  harmonic functions and related topics.
The work was carried out when the first author visited the Department of Mathematical Sciences of the Norwegian University of Science and Technology and the second author visited Chebyshev Laboratory at St.Petersburg State University. We would like to thank both institutions for hospitality.

\end{document}